\documentclass[preprint]{article}%
\usepackage{graphicx}
\usepackage{amsmath,amsxtra,amssymb,latexsym, amscd,amsthm}
\usepackage[mathscr]{eucal}
\newtheorem{thm}{Theorem}[section]

\newtheorem{lem}{Lemma}[section]

\theoremstyle{definition}

\theoremstyle{remark}
\newtheorem{rem}{Remark}[section]
\numberwithin{equation}{section}
\def\ind{{\rm 1\hspace{-0.90ex}1}}
\begin{document}

\title{Kolmogorov distance between the exponential functionals of fractional Brownian motion
}
\author{Nguyen Tien Dung
\footnote{Email: dung\_nguyentien10@yahoo.com}
}

\date{July 20, 2019}          % Comptes Rendus Mathematique

\maketitle
\begin{abstract}
In this note, we investigate the continuity in law with respect to the Hurst index of the exponential functional of the fractional Brownian motion. Based on the techniques of Malliavin's calculus, we provide an explicit bound on the Kolmogorov distance between two functionals with different Hurst indexes.
\end{abstract}

\section{Introduction} Let $B^H=(B^H_t)_{t\in [0,T]}$ be a fractional Brownian motion (fBm) with Hurst index $H\in(0,1).$ We recall that fBm admits the Volterra represention
\begin{equation}\label{densityCIR02}
B^H_t=\int_0^t K_H(t,s)dW_s,
\end{equation}
where $(W_t)_{t\in [0,T]}$ is a standard Brownian motion and  for some normalizing constants $c_H$ and $c'_H,$ the kernel  $K_H$ is given by $ K_{H}(t,s) = c_{H}s^{1/2 -H}\int_{s}^{t}(u-s)^{H-\frac{3}{2}}u^{H-\frac{1}{2}}du$ if $H>\frac{1}{2}$ and
$$K_H(t,s)=c_H\bigg[\frac{t^{H-\frac{1}{2}}}{s^{H-\frac{1}{2}}}(t-s)^{H-\frac{1}{2}}
-(H-\frac{1}{2})\int\limits_s^t\frac{u^{H-\frac{3}{2}}}{s^{H-\frac{1}{2}}}(u-s)^{H-\frac{1}{2}}du\bigg]\,\,\text{if}\,\,H<\frac{1}{2}.$$
%$$ K_{H}(t,s) = c_{H}s^{1/2 -H}\int_{s}^{t}(u-s)^{H-\frac{3}{2}}u^{H-\frac{1}{2}}du\,\,\text{if}\,\,H>\frac{1}{2}.$$
%where $c_H$ is a coefficient depending only on $H.$
Given real numbers $a$ and $\sigma,$ we consider the exponential functional of the form
$$F_H=\int_0^T e^{as+\sigma B^H_s}ds.$$
It is known that this functional plays an important role in several domains. For example, it can be used to investigate the finite-time blowup of positive solutions to semi-linear stochastic partial differential equations \cite{Dozzi2014}.  In the special case $H=\frac{1}{2},$ fBm reduces to a standard Brownian motion and a lot of fruitful properties of $F_{\frac{1}{2}}$ can be founded in the literature, see e.g. \cite{Matsumoto2005a,Matsumoto2005b,Pintoux2010,Yor2001}. In particular, the distribution of $F_{\frac{1}{2}}$ can be computed explicitly. However, to the best our knowledge, it remains a challenge to obtain the deep properties of $F_H$ for $H\neq \frac{1}{2}.$

On the other hand, because of its applications in statistical
estimators, the problem of proving the continuity in law with
respect to $H$ of certain functionals has been studied by several
authors. Among others, we refer the reader to
\cite{Jolis2010,Koch2019,Richard2017,Saussereau2012} and the references therein for the
detailed discussions and the related results. Motivated by this observation, the aim of the present paper is to investigate the continuity in law of the exponential functional $F_H.$ Intuitively, the continuity of $F_H$ with respect to $H$ is not surprising. However, the interesting point of Theorem \ref{tyi3} below is that we are able to give an explicit bound on Komogorov distance between two functionals with different Hurst indexes.
\begin{thm}\label{tyi3}For any $H_1,H_2\in(0,1),$ we have
\begin{equation}\label{hk1}
\sup\limits_{x\geq 0}|P\left(F_{H_1}\leq x\right)-P\left(F_{H_2}\leq
x\right)|\leq C|H_1-H_2|,
\end{equation}
 where $C$ is a positive
constant depending on $a,\sigma,T$ and $H_1,H_2.$
\end{thm}
%\section{Preliminaries}
\section{Proofs}
Our main tools are the techniques of Malliavin calculus. Hence, for the reader's convenience, let us recall some elements of Malliavin calculus with respect to Brownian motion $W,$ where $W$ is used to present $B^H$ as in (\ref{densityCIR02}). We suppose that $(W_t)_{t\in [0,T]}$ is defined on a complete probability space $(\Omega,\mathcal{F},\mathbb{F},P)$, where $\mathbb{F}=(\mathcal{F}_t)_{t\in [0,T]}$ is a natural filtration generated by the Brownian motion $W.$ For $h\in L^2[0,T],$ we denote by $W(h)$ the Wiener integral
$$W(h)=\int\limits_0^T h(t)dW_t.$$
Let $\mathcal{S}$ denote the dense subset of $L^2(\Omega, \mathcal{F},P)$ consisting of smooth random variables of the form
\begin{equation}\label{ro}
F=f(W(h_1),...,W(h_n)),
\end{equation}
where $n\in \mathbb{N}, f\in C_b^\infty(\mathbb{R}^n),h_1,...,h_n\in L^2[0,T].$ If $F$ has the form (\ref{ro}), we define its Malliavin derivative as the process $DF:=\{D_tF, t\in [0,T]\}$ given by
$$D_tF=\sum\limits_{k=1}^n \frac{\partial f}{\partial x_k}(W(h_1),...,W(h_n)) h_k(t).$$
We shall denote by $\mathbb{D}^{1,2}$ the closure of $\mathcal{S}$ with respect to the norm
$$\|F\|^2_{1,2}:=E|F|^2+E\bigg[\int\limits_0^T|D_u F|^2du\bigg].$$
An important operator in the Malliavin calculus theory is the divergence operator $\delta,$ it is the adjoint of the derivative operator $D.$ The domain of $\delta$ is the set of all functions $u\in L^2(\Omega,L^2[0,T])$ such that
$$E|\langle DF,u\rangle_{L^2[0,T]}|\leq C(u)\|F\|_{L^2(\Omega)},$$
where $C(u)$ is some positive constant depending on $u.$ In particular, if $u\in Dom\,\delta,$ then $\delta(u)$ is characterized by the following duality relationship
$$E\langle DF,u\rangle_{L^2[0,T]}=E[F\delta(u)]\,\,\,\text{for any}\,\,\,F\in \mathbb{D}^{1,2}.$$
%for any $F\in \mathcal{S}$ and $u\in L^2(\Omega,L^2[0,T]).$

%Let $F\in \mathbb{D}^{1,2}$ and $u\in Dom\,\delta$ such that  $Fu\in L^2(\Omega,L^2[0,T]).$ Then $Fu\in Dom\,\delta$ and we have the following relation
%\begin{equation}\label{hj44}
%\delta(Fu)=F\delta(u)-\left\langle DF,u\right\rangle _{L^2[0,T]},
%\end{equation}
%provided the right-hand side is square integrable.
In order to be able to prove Theorem \ref{tyi3}, we need two technical lemmas. %Because $(-B^H_t)_{t\in [0,T]}$ is also a fBm, we can and will assume that $\sigma>0.$
\begin{lem}\label{tt5l}For any $H\in(0,1),$ we have $F_H\in \mathbb{D}^{1,2}$ and
$$\left(\int_0^T |D_rF_{H}|^2dr\right)^{-1}\in L^p(\Omega),\,\,\forall\,\,p\geq 1.$$
\end{lem}
\begin{proof}By the representation (\ref{densityCIR02}), we have $D_rB^H_s=K_H(s,r)$ for $0\leq r<s\leq T.$ Hence, $F_H\in \mathbb{D}^{1,2}$ and its derivative is given by
$$D_rF_H=\int_r^T \sigma K_H(s,r)e^{as+\sigma B^H_s}ds,\,\,0\leq r\leq T.$$
So we can deduce
$$D_rF_H\geq e^{-|a| T+\sigma \min\limits_{0\leq s\leq T}B^H_s}\int_r^T \sigma K_H(s,r)ds,\,\,0\leq r\leq T.$$
As a consequence,
\begin{align*}
\int_0^T |D_rF_{H}|^2dr&\geq e^{-2|a| T+2\sigma \min\limits_{0\leq s\leq T}B^H_s}\int_0^T\left(\int_r^T \sigma K_H(s,r)ds\right)^2dr\\
&= \sigma^2 e^{-2|a| T+2\sigma \min\limits_{0\leq s\leq T}B^H_s}\int_0^T\left(\int_r^T  K_H(s,r)ds\right)\left(\int_r^T  K_H(t,r)dt\right)dr\\
&= \sigma^2 e^{-2|a| T+2\sigma \min\limits_{0\leq s\leq T}B^H_s}\int_0^T\int_0^T\left(\int_0^{s\wedge t}  K_H(s,r)K_H(t,r)dr\right)dsdt\\
&= \sigma^2 e^{-2|a| T+2\sigma \min\limits_{0\leq s\leq T}B^H_s}\int_0^T\int_0^TE[B^H_sB^H_t]dsdt\\
&=\frac{T^{2H+2}}{2H+2} \sigma^2 e^{-2|a| T+2\sigma \min\limits_{0\leq s\leq T}B^H_s}.
\end{align*}
In the last equality we used the fact that $E[B^{H}_{t}B^{H}_{s}] = \frac{1}{2}(t^{2H}+s^{2H} - |t-s|^{2H}).$ We therefore obtain
$$\left(\int_0^T |D_rF_{H}|^2dr\right)^{-1}\leq \frac{2H+2}{T^{2H+2}\sigma^2}  e^{2|a| T+2\sigma \max\limits_{0\leq s\leq T}(-B^H_s)}.$$
By Fernique's theorem, we have  $e^{2\sigma \max\limits_{0\leq s\leq T}(-B^H_s)}\in L^p(\Omega)$ for any $p\geq 1.$ This completes the proof.
\end{proof}

%$$F=\int_0^T e^{as+\sigma B_s}ds$$

\begin{lem}\label{tt5la}For any $H_1,H_2\in(0,1),$ we have
\begin{equation}\label{1t1}
E|F_{H_1}-F_{H_2}|^2\leq C|H_1-H_2|^2,
\end{equation}
\begin{equation}\label{1t2}
\int_0^T E|D_rF_{H_1}-D_rF_{H_2}|^2dr\leq C|H_1-H_2|^2,
\end{equation}
where $C$ is a positive constant depending on $a,\sigma,T$ and $H_1,H_2.$
\end{lem}
\begin{proof} By the H\"older inequality we have
\begin{align*}
E|F_{H_1}-F_{H_2}|^2&=E\big|\int_0^T \big(e^{as+\sigma B^{H_1}_s}-e^{as+\sigma B^{H_2}_s}\big)ds\big|^2\\
&\leq T\int_0^T E\big|e^{as+\sigma B^{H_1}_s}-e^{as+\sigma B^{H_2}_s}\big|^2ds.
\end{align*}
Using the fundamental inequality $|e^x-e^y|\leq \frac{1}{2} |x-y|(e^x+e^y)$ for all $x,y$ we deduce
\begin{align*}
E|F_{H_1}-F_{H_2}|^2&\leq \frac{T}{4}\int_0^T E\big|(\sigma B^{H_1}_s-\sigma B^{H_2}_s)(e^{as+\sigma B^{H_1}_s}+e^{as+\sigma B^{H_2}_s})\big|^2ds\\
&\leq \frac{T\sigma^2}{4}\int_0^T \left(E|B^{H_1}_s-B^{H_2}_s|^4\right)^{\frac{1}{2}}\left(E|e^{as+\sigma B^{H_1}_s}+e^{as+\sigma B^{H_2}_s}|^4\right)^{\frac{1}{2}}ds\\
&\leq \frac{T\sigma^2}{4}\int_0^T \left(E|B^{H_1}_s-B^{H_2}_s|^4\right)^{\frac{1}{2}}\left(8E[e^{4as+4\sigma B^{H_1}_s}]+8E[e^{4as+4\sigma B^{H_2}_s}]\right)^{\frac{1}{2}}ds\\
&= \frac{T\sigma^2}{4}\int_0^T \left(E|B^{H_1}_s-B^{H_2}_s|^4\right)^{\frac{1}{2}}\left(8e^{4as+8\sigma^2 s^{2H_1}}+8e^{4as+8\sigma^2 s^{2H_2} }\right)^{\frac{1}{2}}ds\\
&\leq\frac{T\sigma^2}{4}\int_0^T \left(E|B^{H_1}_s-B^{H_2}_s|^4\right)^{\frac{1}{2}}\left(8e^{4as+8\sigma^2 s^{2H_1}}+8e^{4as+8\sigma^2 s^{2H_2} }\right)^{\frac{1}{2}}ds.
\end{align*}
It is known from the proof of Theorem 4 in \cite{Peltier1995} that there exists a positive constant $C$ such that
\begin{equation}\label{1t3}
\sup\limits_{0\leq s\leq T}E|B^{H_1}_s-B^{H_2}_s|^2\leq C|H_1-H_2|^2.
\end{equation}
On the other hand, we have $E|B^{H_1}_s-B^{H_2}_s|^4=3(E|B^{H_1}_s-B^{H_2}_s|^2)^2$ because $B^{H_1}_s-B^{H_2}_s$ is a Gaussian random variable for every $s\in[0,T].$ So we can conclude that there exists a positive constant $C$ such that
$$E|F_{H_1}-F_{H_2}|^2\leq C|H_1-H_2|^2.$$
%\begin{align*}
%E|B^{H_1}_s-B^{H_2}_s|^4&=3\left(\int_0^s |K_{H_1}(s,u)-K_{H_2}(s,u)|^2du\right)^2
%\end{align*}
%$$E[F_H]=\int_0^T e^{as+\frac{1}{2}\sigma^2 s}ds$$
%$$D_rF=\int_r^T\sigma e^{as+\sigma B_s}ds$$
%$$D_rF_{H_1}-D_rF_{H_2}= \sigma \int_r^T\left(K_{H_1}(s,r)e^{as+\sigma B^{H_1}_s}-K_{H_2}(s,r)e^{as+\sigma B^{H_2}_s}\right)ds,\,\,0\leq r\leq T.$$
To finish the proof, let us verify (\ref{1t2}). By the H\"older and triangle inequalities we obtain
\begin{align*}
&E|D_rF_{H_1}-D_rF_{H_2}|^2\leq \sigma^2T \int_r^TE\big|K_{H_1}(s,r)e^{as+\sigma B^{H_1}_s}-K_{H_2}(s,r)e^{as+\sigma B^{H_2}_s}\big|^2ds\\
&\leq 2\sigma^2T \int_r^T|K_{H_1}(s,r)-K_{H_2}(s,r)|^2E[e^{2as+2\sigma B^{H_1}_s}]+K^2_{H_2}(s,r)E\big|e^{2as+2\sigma B^{H_1}_s}-e^{as+\sigma B^{H_2}_s}\big|^2ds,
\end{align*}
and hence,
\begin{align*}
\int_0^TE|D_rF_{H_1}-D_rF_{H_2}|^2dr&\leq 2\sigma^2T \int_0^TE[e^{2as+2\sigma B^{H_1}_s}]\int_0^s|K_{H_1}(s,r)-K_{H_2}(s,r)|^2drds\\
&+2\sigma^2T \int_0^TE\big|e^{2as+2\sigma B^{H_1}_s}-e^{as+\sigma B^{H_2}_s}\big|^2\int_0^sK^2_{H_2}(s,r)drds\\
&= 2\sigma^2T \int_0^Te^{2as+2\sigma^2 s^{2H_1}}E|B^{H_1}_s-B^{H_2}_s|^2ds\\
&+2\sigma^2T \int_0^TE\big|e^{2as+2\sigma B^{H_1}_s}-e^{as+\sigma B^{H_2}_s}\big|^2s^{2H_2}ds.
\end{align*}
Notice that $\int_0^sK^2_{H_2}(s,r)dr=E|B^{H_2}_s|^2=s^{2H_2}.$ Thus the estimate (\ref{1t2}) follows from (\ref{1t1}) and (\ref{1t3}).
\end{proof}
{\bf Proof of Theorem \ref{tyi3}.} For the simplicity, we write $\langle ., .\rangle$ instead of $\langle ., .\rangle_{L^2[0,T]}.$ Borrowing the arguments used in the proof of Proposition 2.1.1 in \cite{nualartm2}, we let $\psi$ be a nonnegative smooth function with compact support, and set $\varphi(y)=\int_{-\infty}^y\psi(z)dz.$ Given $Z\in \mathbb{D}^{1,2},$ we know that $\varphi(Z)$ belongs to $\mathbb{D}^{1,2}$ and making the scalar product of its derivative with $DF_{H_2}$ obtains
$$\langle D\varphi(Z), DF_{H_2}\rangle=\psi(Z)\langle DZ, DF_{H_2}\rangle.$$
Fixed $x\in \mathbb{R}_+,$ by an approximation argument, the above equation holds for $\psi(z)=\ind_{[0,x]}(z).$ Choosing $Z=F_{H_1}$ and $Z=F_{H_2}$ we obtain
$$\langle D\int_{-\infty}^{F_{H_1}}\ind_{[0,x]}(z)dz, DF_{H_2}\rangle=\ind_{[0,x]}(F_{H_1})\langle DF_{H_1}, DF_{H_2}\rangle,$$
$$\langle D\int_{-\infty}^{F_{H_2}}\ind_{[0,x]}(z)dz, DF_{H_2}\rangle=\ind_{[0,x]}(F_{H_2})\langle DF_{H_2}, DF_{H_2}\rangle.$$
Hence, we can get
\begin{align*}
&\langle D\int_{F_{H_2}}^{F_{H_1}}\ind_{[0,x]}(z)dz, DF_{H_2}\rangle=\ind_{[0,x]}(F_{H_1})\langle DF_{H_1}, DF_{H_2}\rangle-\ind_{[0,x]}(F_{H_2})\langle DF_{H_2}, DF_{H_2}\rangle\\
&=\left(\ind_{[0,x]}(F_{H_1})-\ind_{[0,x]}(F_{H_2})\right)\langle DF_{H_2}, DF_{H_2}\rangle+\ind_{[0,x]}(F_{H_1})\langle DF_{H_1}-DF_{H_2}, DF_{H_2}\rangle.
\end{align*}
This, together with the fact that $\|DF_{H_2}\|^2:=\langle DF_{H_2}, DF_{H_2}\rangle>0\,\,a.s.$ gives us
\begin{align*}
\ind_{[0,x]}(F_{H_1})-\ind_{[0,x]}(F_{H_2})&=\frac{\langle D\int_{F_{H_2}}^{F_{H_1}}\ind_{[0,x]}(z)dz, DF_{H_2}\rangle}{\|DF_{H_2}\|^2}-\frac{\ind_{[0,x]}(F_{H_1})\langle DF_{H_1}-DF_{H_2}, DF_{H_2}\rangle}{\|DF_{H_2}\|^2}.
\end{align*}
Taking the expectation yields
\begin{align*}
P\left(F_{H_1}\leq x\right)&-P\left(F_{H_2}\leq x\right)=E[\ind_{[0,x]}(F_{H_1})-\ind_{[0,x]}(F_{H_2})]\\
&=E\left[\int_{F_{H_2}}^{F_{H_1}}\ind_{[0,x]}(z)dz\delta\left(\frac{DF_{H_2}}{\|DF_{H_2}\|^2}\right)\right]-E\left[\frac{\ind_{[0,x]}(F_{H_1})\langle
DF_{H_1}-DF_{H_2}, DF_{H_2}\rangle}{\|DF_{H_2}\|^2}\right]
\end{align*}
By the H\"older inequality
\begin{align*}
\sup\limits_{x\geq 0}|P\left(F_{H_1}\leq x\right)&-P\left(F_{H_2}\leq x\right)| \leq E\bigg|(F_{H_1}-F_{H_2})\delta\left(\frac{DF_{H_2}}{\|DF_{H_2}\|^2}\right)\bigg|+E\bigg|\frac{\langle
DF_{H_1}-DF_{H_2}, DF_{H_2}\rangle}{\|DF_{H_2}\|^2}\bigg|\\
&\leq \left(E|F_{H_1}-F_{H_2}|^2\right)^{\frac{1}{2}}\left(E\delta\left(\frac{DF_{H_2}}{\|DF_{H_2}\|^2}\right)^2\right)^{\frac{1}{2}}+E\bigg|\frac{\|
DF_{H_1}-DF_{H_2}\|}{\|DF_{H_2}\|}\bigg|\\
&\leq \left(E|F_{H_1}-F_{H_2}|^2\right)^{\frac{1}{2}}\left(E\delta\left(\frac{DF_{H_2}}{\|DF_{H_2}\|^2}\right)^2\right)^{\frac{1}{2}}+\left(E\|
DF_{H_1}-DF_{H_2}\|^2\right)^{\frac{1}{2}}\left(E\bigg[\frac{1}{\|DF_{H_2}\|^2}\bigg]\right)^{\frac{1}{2}}.
\end{align*}
Recalling Lemma \ref{tt5la}, we obtain $$\sup\limits_{x\geq
0}|P\left(F_{H_1}\leq x\right)-P\left(F_{H_2}\leq x\right)|\leq
C|H_1-H_2|\left[\left(E\delta\left(\frac{DF_{H_2}}{\|DF_{H_2}\|^2}\right)^2\right)^{\frac{1}{2}}
+\left(E\bigg[\frac{1}{\|DF_{H_2}\|^2}\bigg]\right)^{\frac{1}{2}}\right].$$
Thanks to Lemma \ref{tt5l} we have
$$E\bigg[\frac{1}{\|DF_{H_2}\|^2}\bigg]=E\bigg[\left(\int_0^T |D_rF_{H_2}|^2dr\right)^{-1}\bigg]<\infty.$$
Thus we can obtain (\ref{hk1}) by checking the finiteness of
$E[\delta(u)^2],$ where
$$u_r:=\frac{D_rF_{H_2}}{\|DF_{H_2}\|^2},\,\,0\leq r\leq T.$$
It is known from Proposition 1.3.1 in \cite{nualartm2} that
$$E[\delta(u)^2]\leq \int_0^TE|u_r|^2dr+\int_0^T\int_0^T E|D_\theta u_r|^2d\theta dr.$$
We have
$$\int_0^TE|u_r|^2dr=E\bigg[\frac{1}{\|DF_{H_2}\|^2}\bigg]<\infty.$$
Furthermore, by the chain rule for Malliavin derivative, we have
$$D_\theta u_r=\frac{D_\theta D_rF_{H_2}}{\|DF_{H_2}\|^2}-2\frac{D_rF_{H_2}\langle D_\theta DF_{H_2},DF_{H_2}\rangle}{\|DF_{H_2}\|^4},\,\,0\leq \theta\leq T.$$
Hence, by the H\"older inequality,
\begin{align*}\int_0^T\int_0^T
E|D_\theta u_r|^2d\theta dr&\leq 2E\left[\frac{\int_0^T\int_0^T
|D_\theta D_rF_{H_2}|^2d\theta
dr}{\|DF_{H_2}\|^4}\right]+8E\left[\frac{\int_0^T\int_0^T |D_\theta
D_rF_{H_2}|^2d\theta dr}{\|DF_{H_2}\|^4}\right]\\
&\leq 10\left(E\left[\left(\int_0^T\int_0^T |D_\theta
D_rF_{H_2}|^2d\theta
dr\right)^2\right]\right)^{\frac{1}{2}}\left(E\left[\frac{1}{\|DF_{H_2}\|^8}\right]\right)^{\frac{1}{2}}.
\end{align*}
We now observe that
$$D_\theta D_rF_{H_2}=\int_{r\vee\theta}^T \sigma^2 K_{H_2}(s,r)K_{H_2}(s,\theta)e^{as+\sigma B^{H_2}_s}ds,\,\,0\leq r,\theta\leq T.$$
Hence,
$$|D_\theta D_rF_{H_2}|^2\leq T \sigma^4\int_{r\vee\theta}^T K^2_{H_2}(s,r)K^2_{H_2}(s,\theta)e^{2as+2\sigma B^{H_2}_s}ds,\,\,0\leq r,\theta\leq T$$
and we obtain
$$\int_0^T\int_0^T |D_\theta D_rF_{H_2}|^2d\theta dr\leq T \sigma^4\int_{0}^T s^{4H_2}e^{2as+2\sigma B^{H_2}_s}ds,$$
which implies that
$$E\left[\left(\int_0^T\int_0^T |D_\theta
D_rF_{H_2}|^2d\theta
dr\right)^2\right]\leq T^4 \sigma^8\int_{0}^T s^{8H_2}e^{4as+8\sigma^2 s^{2H_2}}ds<\infty.$$
Finally, we have $E\left[\frac{1}{\|DF_{H_2}\|^8}\right]<\infty$ due to Lemma \ref{tt5l}. So we can conclude that $E[\delta(u)^2]$ is finite. This finishes the proof of Theorem \ref{tyi3}.
\begin{rem}Given a bounded and continuous function $\psi,$ with the exact proof of Theorem \ref{tyi3}, we also have
$$|E[\psi(F_{H_1})]-E[\psi(F_{H_2})]|\leq C|H_1-H_2|.$$
This kind of estimates has been investigated by Richard and Talay for the solution of fractional stochastic differential equations. However, Theorem 1.1 in \cite{Richard2017} requires $H_2=\frac{1}{2}$ and $\psi$ to be H\"older continuous of order $2+\beta$ with $\beta>0.$
\end{rem}

\noindent {\bf Acknowledgments.}  This research was funded by Vietnam National Foundation for Science and Technology Development (NAFOSTED) under grant number 101.03-2019.08.

\end{document}